\documentclass[11pt,english,final]{amsart}
\usepackage[T1]{fontenc}
\usepackage[latin9]{inputenc}
\usepackage{babel}
\usepackage{prettyref}
\usepackage{amsthm}
\usepackage{amsmath}
\usepackage{amssymb}
\usepackage{mathrsfs}
\usepackage{esint}
\usepackage{hyperref}
\usepackage[normalem]{ulem}

\usepackage{color}

\makeatletter
\numberwithin{equation}{section}
\numberwithin{figure}{section}
\theoremstyle{plain}
\newtheorem{thm}{\protect\theoremname}
  \theoremstyle{remark}
  \newtheorem{rem}[thm]{\protect\remarkname}
  \theoremstyle{plain}
  \newtheorem{prop}[thm]{\protect\propositionname}
  \theoremstyle{definition}
  \newtheorem*{problem*}{\protect\problemname}
  
  \newtheorem*{question*}{Question}
  \theoremstyle{plain}
  \newtheorem{lem}[thm]{\protect\lemmaname}
  \theoremstyle{plain}
  \newtheorem{cor}[thm]{\protect\corollaryname}
  \theoremstyle{plain}
  \newtheorem{fact}[thm]{\protect\factname}

\usepackage{fullpage}

\newcommand{\cE}{\mathcal{E}}

\newcommand{\cL}{\mathcal{L}}

\newcommand{\cB}{\mathcal{B}}
\newcommand{\cS}{\mathcal{S}}
\newcommand{\R}{\mathbb{R}}

\newcommand{\N}{\mathbb{N}}

\newcommand{\prob}{\mathbb{P}}
\newcommand{\E}{\mathbb{E}}

\newcommand{\eqdist}{\stackrel{(d)}{=}}

\newcommand{\abs}[1]{\lvert#1\rvert}
\newcommand{\norm}[1]{\lvert\lvert#1\rvert\rvert}

\newrefformat{prop}{Proposition~\ref{#1}}
\newrefformat{cor}{Corollary~\ref{#1}}
\newrefformat{sub}{Section~\ref{#1}}
\newrefformat{eq}{Eq.~\eqref{#1}}
\newrefformat{fact}{Fact~\ref{#1}}
\newrefformat{rem}{Remark~\ref{#1}}

\makeatother

  \providecommand{\corollaryname}{Corollary}
  \providecommand{\factname}{Fact}
  \providecommand{\lemmaname}{Lemma}
  \providecommand{\problemname}{Problem}
  \providecommand{\propositionname}{Proposition}
  \providecommand{\remarkname}{Remark}
\providecommand{\theoremname}{Theorem}

\begin{document}

\title{On the Spectral Gap of Spherical Spin Glass Dynamics}

\author{Reza Gheissari}
\address{R.\ Gheissari\hfill\break
Courant Institute\\ New York University\\
251 Mercer Street\\ New York, NY 10012, USA.}
\email{reza@cims.nyu.edu}

\author{Aukosh Jagannath}
\address{A.\ Jagannath\hfill\break
Department of Mathematics\\ University of Toronto\\
40 St George Street\\ Toronto, ON, Canada}
\email{aukosh@cims.nyu.edu}

\maketitle
\vspace{-.5cm}
\begin{abstract}
We consider the time to equilibrium for the Langevin dynamics of the
spherical $p$-spin glass model of system size $N$. We show that
the log-Sobolev constant and spectral gap are order $1$ in $N$
at sufficiently high temperatures whereas the spectral gap decays exponentially
in $N$ at sufficiently low temperatures. These verify the existence
of a dynamical high temperature phase and a dynamical glass phase at the level of the spectral gap. Key to these results are the understanding of
the extremal process and restricted free energy of Subag--Zeitouni
and Subag.
\end{abstract}
\smallskip

\section{Introduction}

In the study of glassy systems such as spin glasses and structural
glasses \cite{BerBir11,CriLeu04,MPV87} and constraint satisfaction
problems \cite{DSS15,MMZ01,MezMont09,MPV87}, one of the fundamental
objects of study is the time to relax to equilibrium. It is believed
that natural dynamics for such systems undergo what is called a \emph{glass
transition} but the nature of such a transition is still unresolved in condensed
matter physics \cite{BerBir11,DebStil01}. At high temperature, one
expects the system to reach equilibrium quickly as it is in
a classical phase, e.g., paramagnetic. At low temperature, however,
when the system is in a dynamical glassy phase, the equilibration
time is expected to be far longer than observable timescales \cite{BerBir11}.
It is desirable to have a mathematically rigorous understanding of
how these timescales to equilibrium change with temperature in well-studied
models. In this paper, we rigorously study the timescales to equilibrium
for an archetypal glassy model, namely the \emph{spherical p-spin
glass model}, defined as follows.

The state space for the spherical $p$-spin glass is the $(N-1)$-sphere
in dimension $N$ of radius $\sqrt{N}$, 
\[
\cS^{N}=S^{N-1}(\sqrt{N})=\bigg\{ \sigma=(\sigma_{1},...,\sigma_{N})\in\mathbb{R}^{N}:\sum_{i=1}^{N}\sigma_{i}^{2}=N\bigg\} \,,
\]
equipped with the induced metric $g$. For $p\geq3$, define the\emph{
$p$-spin Hamiltonian} by, 
\begin{equation}
H_{N,p}(\sigma)=\frac{1}{N^{(p-1)/2}}\sum_{i_{1},\ldots,i_{p}=1}^{N}J_{i_{1}\ldots i_{p}}\sigma_{i_{1}}\ldots\sigma_{i_{p}}\,,\label{eq:p-spin}
\end{equation}
where $J_{i_{1},\ldots,i_{p}}$ are i.i.d. standard Gaussian random
variables. Throughout this paper we will drop the subscripts $p$
and $N$ when it is unambiguous. Corresponding to $H$, define the
\emph{Gibbs measure,} $\pi_{N}$, at inverse temperature $\beta>0$
by 
\[
d\pi_{N}(\sigma)=\frac{e^{-\beta H}}{Z}dV(\sigma)\,,
\]
here $dV$ is the normalized volume measure, and $Z$ is chosen
so that $\pi_{N}$ is a probability measure. Define the \emph{Langevin
dynamics }as the heat flow $$P_{t}=e^{t\mathcal{L}_{N}}$$ generated
by the operator, 
\begin{equation}
\mathcal{L}_{N}=\frac{1}{2}\left(\Delta-\beta g(\nabla H_{N},\nabla\cdot)\right)\,,\label{eq:generator}
\end{equation}
where $\nabla$ is the covariant derivative, and $\Delta$ is the
corresponding Laplacian. In more probabilistic terms, $\mathcal{L}_{N}$
is the infinitesimal generator of a reversible Markov process whose
invariant measure is $\pi_{N}$. (For a quick review of the properties
of $\cL$ see \prettyref{sec:Spectral-Gap-inequalities}.)

One of the defining features of spin glasses is the complexity of
their energy landscape: they generally have exponentially many critical
points that are separated by energy barriers of height diverging linearly
in $N$. Although this complexity leads to rich phenomenological
behavior, it is also at the heart of the difficulty of analyzing these
systems. Indeed, even making this picture rigorous is a difficult
problem. In our setting, it has been established rigorously
in \cite{ABA13,ABC13} for all $p\geq3$. 

Dynamically, the models are expected to have the following rich
behavior that is a hallmark of dynamics for glassy systems. At small
$\beta$, they are expected to be in the \emph{high temperature phase}
where $P_{t}$ behaves similarly to the heat semigroup for the Laplacian on $\cS^{N}$.
For large $\beta$, however, this comparison breaks down and the system
enters \emph{the glass phase}. Here it is believed that $P_{t}$ exhibits
exponentially slow in $N$ relaxation to equilibrium and \emph{aging}
(see the literature review below). A natural question, and the aim
of this paper, is to make the relaxation picture rigorous.

A canonical way to analyze this from the point of view of Markov processes
is through the analysis of the \emph{spectral gap}, that is, the first
nontrivial eigenvalue, called $\lambda_{1}$, of $-\cL$, which governs
the time to equilibrium (see Subsection 1.1). Here the goal is to
analyze the asymptotics of $\lambda_{1}$ in $N$ as we vary $\beta$.
From this framework the above expectation is natural as one expects
\emph{metastable }behavior leading to poor mixing due to the large
energy barriers at low temperature (see e.g., Arrhenius's law).
In the non-disordered setting, there is a vast and growing literature
following this approach: central to this field is the differentiation
of high and low temperature phases where the dynamics moves from
an order $1$ gap to an exponentially decaying gap. This phenomenon
has been observed in lattice systems such as the 2D Ising model (see
e.g., \cite{AiHo,DoSh85,Ho85,MaOl94,StrZeg92Dob}), and
in mean field models including the Curie-Weiss model \cite{BovDen2015,GWL,LLP10}.

The study of the spectral gap for natural spin glass dynamics has
a much more limited history, though similar transitions are expected. 
For the ``simplest'' mean-field model of spin glasses, the random
energy model (REM), it was found that there is only one dynamical phase
in the natural local dynamics \cite{FIKP98}. For models on the hypercube, there is an exponential lower bound on the spectral gap in terms of an intrinsic quantity~\cite{Mathieu}.
In the short range setting, there are some results from e.g.,~\cite{GuZe96,Des02}.
However, for the
classical mean-field models of the $p$-spin models
on $\{\pm1\}^{N}$ and $\mathcal{S}^{N}$, the study of the spectral gap of Glauber/Langevin dynamics has remained largely open.

In the mean-field spin glass dynamics literature, a different approach
has been utilized to analyze off-equilibrium dynamics of the system.
The aim here is to establish a set of equations for the evolution
of certain observables in the large $N$ limit---called the \emph{Cugliandolo-Kurchan
equations}---and observe a transition in the large $t$ behavior
as one varies $\beta$ (see \cite{CugKur93}). At low temperatures,
this leads to the development of the theory of \emph{aging}. The Cugliandolo--Kurchan equations were proven by Ben Arous, Dembo and Guionnet \cite{BADG01,BADG06} for a ``soft''
relaxation of spherical $p$-spin glass dynamics; furthermore, in the case $p=2$ this led to a proof of aging~\cite{BADG01}.  
At high
temperature the same problem was studied as the relaxation goes to
zero in \cite{DGM07}, and similar analyses were undertaken in the
study of related models in \cite{BAG1,BAG2}. Such studies of off-equilibrium
dynamics are restricted to time scales shorter than the relaxation
time of the dynamics. 
Aging has also been extensively studied in related settings on the hypercube. In the REM, aging was established for the random hopping time dynamics, a randomly trapped random walk, in~\cite{ BABoGa2,BABoGa}, in a local Glauber-type dynamics~\cite{MaMo15}, and more recently Metropolis dynamics~\cite{CernyWassmer,Gayr16}.  For the $p$-spin model on $\{\pm 1\}^N$, aging was studied, again for the random hopping time dynamics, in~\cite{BABoCe,BAGu,BoGa, BoGaSv}.

In this paper, we demonstrate, for the relaxation time, the existence
of a dynamical high temperature and dynamical low temperature glass
phase in the setting of Langevin dynamics for spherical $p$-spin glasses. In particular,
we show that the spectral gap of $-\cL_{N}$, has order $1$ asymptotics
in $N$ for $\beta$ small and exponentially decaying in $N$ asymptotics for
$\beta$ large.

\subsection{Statement of Main Results.}

The goal of this paper is to study the behavior of the spectral gap
of the infinitesimal generator, $\cL$ defined in \prettyref{eq:generator},
of the Langevin dynamics for the $p\geq3$ spherical spin glass model.
Observe that $-\cL$ is a non-negative essentially self-adjoint operator on
$C^{\infty}(\cS^{N})\subset L^{2}(dV)$ and has pure point spectrum
$0=\lambda_{0}\leq\lambda_{1}\leq\lambda_{2}\leq\ldots$ ; we point
the reader to \prettyref{sec:Spectral-Gap-inequalities} for a brief
sketch of these facts. 

The asymptotic rate of growth of $\lambda_{1}$ in $N$ is of particular
interest, as $\lambda_{1}^{-1}$, called the \emph{relaxation time,}
is a measure of the time to equilibrium in an $L^{2}$ sense. Our
main result is to show that for all $p\geq3$ , the spectral gap of
the pure spherical $p-$spin model dynamics is in a dynamical high
temperature phase for small $\beta$ and is in a dynamical glass phase
for large $\beta$, suggesting the existence of a \emph{dynamical
glass transition} for the relaxation time:
\begin{thm}
\label{thm:spectral gap}For any $p\geq3$, consider the Langevin
dynamics of the pure spherical $p$-spin glass model at inverse temperature
$\beta>0$ with generator $\mathcal{L}$. 
\begin{enumerate}
\item There exists $0<\beta_{l}(p)<\infty$ and constants $c_{1}(p,\beta),c_{2}(p,\beta)>0$ such that for all $\beta>\beta_{l}$,
\[
\lim_{N\to\infty}\mathbb{P}(c_{1}<-\frac{1}{N}\log\lambda_{1}<c_{2})=1\,.
\]

\item There exists a $\beta_{h}(p)>0$ and a constant $c_3 (p,\beta)>0$ such that for all $\beta<\beta_{h}$,
\begin{align*}
\lim_{N\to\infty}\mathbb{P}(\lambda_{1}>c_{3}) & =1\,.
\end{align*}

\end{enumerate}
\end{thm}

\begin{rem}
It is worth noting here that in the above, (1) holds for all $\beta$ larger than the $\beta_l$ necessary for the results of~\cite{SubGibbs16} to hold; in particular, that picture is expected to hold up to the static phase transition point $\beta_s$. Precise information about the relation between the constants $c_1,c_2$ in (1) and their dependence on $\beta$ can be gleaned from the proofs, though the two do not match. 
\end{rem}

At the heart of the proof of item (1) are the
recent results regarding the energy landscape, $H$, and
the Gibbs measure, $\pi$, developed in a series of papers by Auffinger-Ben
Arous-Cerny \cite{ABC13}, Auffinger-Ben Arous \cite{ABA13}, Subag-Zeitouni
\cite{SubZeit16}, and Subag \cite{Sub15,SubGibbs16}. In particular,
the proof of part (1) of \prettyref{thm:spectral gap} relies on the
restricted free estimates obtained by Subag \cite{SubGibbs16}
(see \prettyref{prop:(Subag)-band} below) in the recent study of
the geometry of the Gibbs measure in spherical $p$-spin models. 

The proof of item (2) follows from the following stronger result, namely
that at high temperature, $\pi$ admits a logarithmic Sobolev (log-Sobolev)
inequality (see~\eqref{eq:log-sobolev-def}).
\begin{prop}
\label{prop:log-sob-ht}There exists a $\beta_{h}(p)>0$ and a constant $c_L(p,\beta)>0$ such that
for all $\beta<\beta_{h}$, $\pi$ admits a log-Sobolev inequality with constant $c_{L}$
with probability $1-O(e^{-cN})$ for some $c>0$.
\end{prop}
\begin{rem}
The proof of \prettyref{prop:log-sob-ht} and therefore item (2) of Theorem~\ref{thm:spectral gap} also goes through for mixed $p$-spin glasses on $\mathcal S^N$.
\end{rem}
This result does not follow by a tensorization argument as is common
for short-range spin systems because $H$ is non-local and $\cS^{N}$
is not a product space. Instead it follows by curvature dimension
arguments after proving that the Hessian of the Hamiltonian is on
the same order of magnitude as the Ricci tensor, uniformly over $\mathcal S^N$; this follows by Gaussian comparison techniques.

Aside from its inherent interest, this also yields the following geometric
analytic interpretation of \prettyref{thm:spectral gap}. For $\beta$
small, the curvature dimension of the system is positive and order
$1$, so that the effective geometry admits a comparison to Gaussian/spherical
space. At low temperature, however, the energetic effects dominate
and thus this comparison breaks down. One is then in a regime where
the time to equilibrium is governed by passing between energy barriers.

\begin{rem}
\label{rem:p=2}The definition of $H_{N,p}$ extends naturally
to $p=2,$ sometimes called the spherical Sherrington-Kirkpatrick model; we omit this case for the
following reason. In contrast to all $p\geq3$, the $p=2$ Hamiltonian
has exactly $N$ critical points, yielding a very different structure
to the energy landscape. The absence of exponentially many metastable
states, a signature of the glassy phase, makes the $p=2$ case less
pertinent to the scope of this paper.
\end{rem}

\subsection*{Phase Boundaries in $\beta$. }

In light of the main theorem, it is natural to define the following
two inverse temperatures. Let
\begin{align*}
\beta_{para}= & \sup\left\{ \beta>0:\lim_{N\to\infty}\mathbb{P}(\lambda_{1}\asymp1)=1\right\} \\
\beta_{dyn}= & \inf\left\{ \beta>0:\lim_{N\to\infty}\mathbb{P}(-\frac{1}{N}\log\lambda_{1}\asymp1)=1\right\} 
\end{align*}
where $f(N)\asymp1$ is to say there exist, $c,C>0$ depending on $p$ and $\beta$ such that $c<f(N)<C$. These correspond
to the thresholds for the dynamical high temperature and glassy phases,
as discussed in the introduction. Evidently $\beta_{h}\leq\beta_{para}$
and $\beta_{dyn}\leq\beta_{l}.$ We are led to the following
question:
\begin{question*}
Is $\beta_{dyn}=\beta_{para}$? 
\end{question*}
\noindent We expect that the equality is true, though we believe our method
for part (1) of the theorem can only be extended to $\beta\geq\beta_{s}$
(where $\beta_{s}$ is the static transition temperature obtained
in \cite{TalSphPF06}), because it relies heavily on information about
the equilibrium measure in the static low temperature regime. 

It is also natural to ask the question of whether the dynamical glass
phase and the static low temperature (glass) phases are in fact distinct.
\begin{question*}
Is $\beta_{dyn}<\beta_{s}$?
\end{question*}
\noindent The answer to this question is expected to be yes \cite{CastCav05,CugKur93}. 

Bearing in mind the results of \cite{BADG01,BADG06} where they define
a critical temperature for the aging phenomena, $\beta_{aging}$ for
a relaxation of the spherical $p$-spin model, it would also be interesting
to prove the existence of aging for large but finite $N$ in the spherical
$p$-spin glass and determine the relation between $\beta_{aging}$,
and the static and dynamical critical temperatures, $\beta_{s}$ and
$\beta_{dyn}$.

\subsection*{Acknowledgements}
The authors thank the anonymous referee for helpful comments and suggestions. 
R.G. would like to thank Eyal Lubetzky and Charles Newman for their support. A.J. would like to thank Dmitry Panchenko
and G\'erard Ben Arous for helpful discussions. This research was conducted
while R.G. was supported by NSF DMS-1207678 and while A.J. was supported
by NSF OISE-1604232.

\section{Preliminaries \label{sec:Preliminaries}}

In this section, we discuss basic properties of the energy landscape,
$H_{N}$. We will prove an important regularity estimate regarding
the operator norm of the Hessian of $H_{N}$ to show that it is uniformly (over $\mathcal S^N$) order one. In particular, this regularity estimate (Lemma~\ref{lem:spectral-radius-bound}) will be the crux of the proof of item (2) of Theorem~\ref{thm:spectral gap}.

We will then proceed recall
results and notation from \cite{ABA13,ABC13,SubGibbs16,SubZeit16}
that will be important to the proofs of item (1) of Theorem~\ref{thm:spectral gap}.

\subsubsection*{Notation\label{sub:Notation}}

In the following we drop the subscripts $p,N$ whenever it is unambiguous,
and we extend the definition of $H_{N,p}$ to $p=1,2$ in the natural
way, when necessary. We say that $f(N)\lesssim_{a}g(N)$ if there
is a constant $C(a)$ that depends only on $a$ such that $f\leq Cg$
for all $N$. Whenever we use the notation $o(1)$, we mean by $f(N)=o(g(N))$ that $f(N)/g(N) \to 0$ as $N\to\infty$. 

For a probability measure $\mu$ let $L_{\mu}^{2}$
denote the space of functions that are square integrable with respect
to $\mu$. Let $C^{\infty}(M)$ be the space of smooth functions on
a Riemannian manifold $M$. The notation $\nabla$ will always refer
to a covariant derivative and $\Delta$ the corresponding Laplacian.

Throughout the paper, let $R(\sigma,\sigma')$ be the normalized spin overlap: for $\sigma,\sigma'\in \mathcal S^N$,
\[
R(\sigma, \sigma') = \frac 1N \sum_i \sigma_i \sigma'_i\,.
\]
Notice that $\mathbb E[H_{N,p}(\sigma) H_{N,p}(\sigma')] =N R(\sigma,\sigma')^p$.

\subsection{Regularity of $H$.}
Before proving the uniform bound on the Hessian of $H$, we remind the reader that the maximum and minimum of the process $H$ are order $N$. \begin{lem}
\label{lem:(Auffinger-Ben-Arous)} For every $p\geq 1$, there exists $E(p)>0$ and $c(p)>0$, such that for every
$\delta>0$, 
\[
\mathbb{P}\left(\max_{\sigma\in\mathcal{S}^{N}}H(\sigma)-NE\geq N\delta\right)\lesssim e^{-cN\delta^2}.
\]
In particular, for every $p\geq1$, we have, 
\begin{align*}
\mathbb{E}\left[\max_{\sigma\in\mathcal{S}^{N}}|H(\sigma)|\right]\lesssim_{p} N\,.
\end{align*}
\end{lem}
The proof of the bound on $\mathbb E[\max_{\cS^N} H]$ (and by symmetry also $\mathbb E[\max_{\cS^N} |H|]$)   in Lemma~\ref{lem:(Auffinger-Ben-Arous)} is a classical application of Dudley's entropy integral; the tail estimate above then follows immediately from Borell's inequality~\cite{Led01}.

\begin{rem}\label{rem:explicit-E_0}
The precise constant, call it $E_0(p)$, such that $\mathbb E[\min H]= -E_0 N +o(N)$ was identified by Auffinger, Ben-Arous and Cerny~\cite{ABC13} (see also~\cite{Sub15}).
Namely, in \cite[Theorem 2.12]{ABC13}, it is stated for $p$ even
as, at the time, the free energy had only been computed for those
$p$'s rigorously. This has been done now by \cite{Chen13} for all
$p\geq3$ so the proof of \cite[Theorem 2.12]{ABC13} holds for all
$p\geq3$. For
$p=2$, the estimate comes from the top eigenvalue of a GOE matrix
\cite{AGZ10}. 
\end{rem}

We now turn to the estimate regarding the Hessian of $H$, central to the proof of item (2) of Theorem~\ref{thm:spectral gap}. In the following,
for $f\in C^{2}$, we let $Hess(f(\sigma))$ denote the covariant
Hessian of $f$ with respect to $\cS^{N}$ at the point $\sigma$,
and $Hess_{E}$ denote the usual Euclidean Hessian on $\R^{N}$. Recall
that the tangent space to $\cS^{N}$ at a point $\sigma$ can then
be thought of as the vector space $\{x\in\R^{N}:(x,\sigma)_{E}=0\}$
where by $(\cdot,\cdot)_{E}$ we mean the usual Euclidean inner product.
With this in mind, for $f\in C^{2}(\R^{N})$ we have that at any point
$\sigma$, 
\begin{equation}
Hess(f(\sigma))=Hess_{E}(f(\sigma))-\frac{1}{N}(\sigma,\nabla_{E}f(\sigma))_{E}Id\label{eq:spherical-euclidean-hessian}
\end{equation}
where $\nabla_{E}$ is the Euclidean gradient, and $(\cdot,\cdot)_{E}$
is the usual Euclidean inner product in $\mathbb{R}^{N}$, and $Id$
is the identity operator on $T_{\sigma} \cS^N$. Define now the quantities
\[
\overline{r}(H)=\sup_{\sigma\in\cS^{N}}\sup_{\substack{v\in T_{\sigma}\cS^{N}\\
g(v,v)=1
}
}Hess(H(\sigma))(v,v)
\]
and 
\[
\underline{r}(H)=\inf_{\sigma\in\cS^{N}}\inf_{\substack{v\in T_{\sigma}\cS^{N}\\
g(v,v)=1
}
}Hess(H(\sigma))(v,v).
\]
By separability of $T\cS^{N}$ and the continuity of $H$, these random
variables are measurable. Furthermore, by symmetry, 
\[
-\overline{r}(H)\eqdist\underline{r}(H).
\]
Finally, define $r(H)=\overline{r}-\underline{r}$. Observe that $r(H)$
bounds the spectral radius of $Hess(H(\sigma))$ uniformly over $\sigma\in\mathcal{S}^{N}$. 
\begin{lem}
\label{lem:spectral-radius-bound}For any $p\geq3$, we have that
\[
\E\left[r(H)\right]\lesssim_{p}1
\]
and there exists a $c(p)>0$ such that for all $\epsilon>0$, 
\[
\prob\left(\abs{r(H)-\E\left[r(H)\right]}>\epsilon\right)\lesssim e^{-cN\epsilon^{2}}\,.
\]
\end{lem}
\begin{proof}
By symmetry it suffices to prove the estimates for $\bar{r}(H)$.
We begin by proving the first estimate. To this end, observe that
$H$ can be extended to all of $\R^{N}$ by allowing $\sigma$ to
take values in $\mathbb{R}^{N}$ and using the same definition of
the Hamiltonian. Thus in the notation above, 
\begin{align*}
(\sigma,\nabla_{E}H(\sigma))_{E}= & pH(\sigma).
\end{align*}
Combining this with \prettyref{eq:spherical-euclidean-hessian} and
the fact that $H$ is smooth, we then see that for any $v\in S^{N-1}(1)\subset\mathbb{R}^{N}$,
we have 
\[
Hess(H_{N}(\sigma))(v,v)=\frac{p(p-1)}{N^{\frac{p-1}{2}}}\sum_{l,m,i_{1},...,i_{p-2}=1}^{N}J_{l,m,i_{1}\ldots,,i_{p-2}}\sigma_{i_{1}}\cdots\sigma_{i_{p-2}}v_{l}v_{m}-\frac{p}{N}H_{N}(\sigma)\norm{v}_{{2}}^{2}\,,
\]
when viewed as an operator on $T_{\sigma} \cS^N$. 

Define the $\cS^{N}\times S^{N-1}(1)$-indexed Gaussian process, $\psi(\sigma,v)$,
given by 
\[
\psi(\sigma,v)=\frac{p(p-1)}{N^{\frac{p-1}{2}}}\sum_{l,m,i_{1},...,i_{p-2}=1}^{N}J_{l,m,i_{1}\ldots,,i_{p-2}}\sigma_{i_{1}}\cdots\sigma_{i_{p-2}}v_{l}v_{m}-\frac{p}{N}H_{N}(\sigma)\,.
\]
As $\cS^{N}$ is given by induced metric, we have 
\[
\overline{r}=\sup_{\sigma\in\cS^{N}}\sup_{v\in S^{N-1}(1)\cap T_\sigma \cS^N}\psi\leq\sup_{\sigma\in\cS^{N}}\sup_{v\in S^{N-1}(1)}\psi.
\]
Define also the related process 
\[
\phi(\sigma,v)=\frac{p(p-1)}{N^{\frac{p-1}{2}}}\sum_{i_{1},...,i_{p-2}=1}^{N}J'_{i_{1}\ldots,,i_{p-2}}\sigma_{i_{1}}\cdots\sigma_{i_{p-2}}+\frac{p(p-1)}{\sqrt{N}}\sum_{l,m=1}^{N}J''_{lm}v_{l}v_{m}-\frac{p}{N}H_{N}(\sigma)\,,
\]
where $J'_{i_{1}\ldots i_{p-2}}$ and $J''_{lm}$ are independent standard
Gaussians. For any $\sigma,\sigma'\in\mathcal{S}^{N}$, $v,v'\in S^{N-1}(1)$,
one sees that,
\begin{align*}
\mathbb{E}(\psi(\sigma,v)-\psi(\sigma',v'))^{2}\leq & \frac{2p^{2}}{N^{2}}\mathbb{E}(H_{N}(\sigma)-H_{N}(\sigma'))^{2}+\frac{2p^{2}(p-1)^{2}}{N^{p-1}}\sum\left(\sigma_{i_{1}}\cdots\sigma_{i_{p-2}}\right)^{2}(v_{l}v_{m}-v'_{l}v'_{m})^{2}\\
 & +\frac{2p^{2}(p-1)^{2}}{N^{p-1}}\sum\left((\sigma_{i_{1}}\cdots\sigma_{i_{p-2}}-\sigma_{i_{1}}'\cdots\sigma_{i_{p-2}}')v_{l}v_{m}\right)^{2}\,.
\end{align*}
where the above sums are over $l,m,i_{1},...,i_{p-2}\in[N]$. The
first term we leave as is and bound the sum of the latter two terms:
\begin{align*}
\frac{1}{N^{p-1}}\sum\left(\sigma_{i_{1}}\cdots\sigma_{i_{p-2}}\right)^{2}(v_{l}v_{m}-v'_{l}v'_{m})^{2}\lesssim_{p} & \frac{1}{N}\sum_{l,m=1}^{N}(v_{l}v_{m}-v'_{l}v_{m}')^{2}\,,
\end{align*}
and similarly,
\begin{align*}
\frac{1}{N^{p-1}}\sum\left((\sigma_{i_{1}}\cdots\sigma_{i_{p-2}}-\sigma_{i_{1}}'\cdots\sigma_{i_{p-2}}')v_{l}v_{m}\right)^{2}\lesssim_{p} & \frac{1}{N^{p-1}}\sum_{i_{1},...,i_{p-2}=1}^{N}\left(\sigma_{i_{1}}\cdots\sigma_{i_{p-2}}-\sigma_{i_{1}}'\cdots\sigma_{i_{p-2}}'\right)^{2}\,.
\end{align*}
Putting this together, we see that for any $\sigma,\sigma'\in\mathcal{S}^{N}$,
$v,v'\in S^{N-1}(1)$, 

\[
\E\left(\psi(\sigma,v)-\psi(\sigma',v')\right)^{2}\lesssim_{p}\E\left(\phi(\sigma,v)-\phi(\sigma',v')\right)^{2}\,.
\]
Thus by the Sudakov-Fernique inequality \cite{LedouxTalagrand}, we
have that 
\begin{align*}
\E\left[\overline{r}(H)\right]& = \E\left[\sup_{\sigma\in\mathcal{S}^{N}}\sup_{v\in S^{N-1}(1)}\psi(\sigma,v)\right]\\
&\lesssim_{p} \E\left[\sup_{\sigma\in\mathcal{S}^{N},v\in S^{N-1}(1)}\phi(\sigma,v)\right]\\
&\lesssim_{p} \frac{1}{N}\E\left[\sup_{x\in\mathcal{S}^{N}}H_{N,p-2}(x)+\sup_{x\in\cS^{N}}H_{N,2}(x)+\sup_{x\in\cS^{N}}H_{N,p}(x)\right]\\
&\lesssim_{p} 1\,.
\end{align*}
The second to last inequality comes from scaling $v$, and the last
inequality is a direct consequence of \prettyref{lem:(Auffinger-Ben-Arous)}.
Thus we have the first inequality in~\prettyref{lem:spectral-radius-bound}.

We now turn to proving the second inequality. To this end observe
that for every $\sigma\in\mathcal{S}^{N}$, $v\in S^{N-1}(1)$, we
have that
\[
\mathbb{E}[\psi(\sigma,v)^{2}]\leq\frac{2p^{2}(p-1)^{2}}{N^{p-1}}\sum_{l,m,i_{1},...,i_{p-2}=1}^{N}(\sigma_{i_{1}}\cdots\sigma_{i_{p-2}}v_{l}v_{m})^{2}+\frac{2}{N^{2}}\mathbb{E}[H_{N,p}(\sigma)^{2}]\lesssim_{p}\frac{1}{N}\,.
\]
The result then follows by Borell's inequality \cite{Led01}.
\end{proof}

\subsection{Previous Results\label{sub:Previous-results}}
We now remind the reader of several  recent results that 
give a good understanding of the critical points of $H$ with near-minimal
energy. These will be important to the proof of item (1) of~\prettyref{thm:spectral gap}.

We begin by observing that the conditional law of $H$ in a neighborhood of a critical point
has a simple explicit form in terms of other $p$-spin models.
This result follows by direct calculations as can be seen, for example in \cite{SubGibbs16}. We
state the result in the weakest form that we need. For each $x\in\mathcal{S}^{N}$
define the following conditional measure,
\begin{align*}
\prob_{u}(\cdot)= & \mathbb{P}(\cdot\mid H(x)=u,\nabla H\restriction_{x}=0)\,,
\end{align*}
with corresponding expectation $\E_{u}$, where the dependence on
$x$ is implicit. Dropping the dependence on $x$ is justified as
this law is invariant in $x$ by isotropy. Evidently, this is the
law of $H$ conditioned on the event that $x$ is a critical point of $H$
with energy $u$.
\begin{lem}
\label{lem:H-reform} Let $u\in\R$ and $x\in\cS^{N}$. Then, with
respect to $\prob_{u}$, $H_{N}(\sigma)$ satisfies 
\[
H_{N}(\sigma)\eqdist uR(\sigma,x)^{p}+Y_{N}(\sigma)\,,
\]
where $Y_{N}(\sigma)$ is a centered, smooth Gaussian process satisfying,
\begin{align*}
\mbox{Cov}(Y_{N}(\sigma),Y_{N}(\sigma')) & =Nf(R(\sigma,\sigma'))\,,\qquad\mbox{and}\\
\E\left[\max_{\sigma\in\cS^{N}}Y_{N}(\sigma)\right] & <\infty\,,
\end{align*}
where $f$ is a polynomial of degree $p$ whose coefficients depend
only on $p$.\end{lem}
\begin{proof}
Recall that $(H_N(\sigma),\nabla H_N(\sigma))$ are jointly Gaussian. 
The distributional equality then follows by computing the conditional law
of $H$ given $\nabla H(x)$ and $H(x)$. See, for example, \cite[Lemmas~14--15]{SubGibbs16}. Since $Y$ is
a.s.\ a continuous Gaussian process on a compact space, $\max_{\mathcal{S}^{N}}Y_{N}$
is a.s.\ finite. The last result follows from this, the covariance estimate and Borell's inequality
(see, e.g., \cite{Led01}).
\end{proof}

In the subsequent, it will be useful to understand basic properties
of the local minima of the Hamiltonian. To this end, we introduce
the following notation regarding the critical points of $H.$ Observe
that $H$ is smooth, and almost surely Morse. (A function is Morse if its critical points are non-degenerate.) 
Furthermore, it has a global minimum that is a.s.\ unique for $p$ odd and unique modulo the
reflection symmetry $\sigma\mapsto-\sigma$ for $p$ even, where we
note that every smooth real-valued function on the sphere has finitely
many critical points.

A natural question is to count the expected number of critical points of $H$.
This was studied in~\cite{ABC13}.
Let 
\[
\Theta_p(E)=\lim_{N\to\infty} \frac 1N \log \mathbb E[|\{x:\nabla H(x)=0, H(x)\leq EN\}|].
\]
In \cite{ABC13}, it was shown that $\Theta_p(E)$ has the following explicit form.
\[
\Theta_{p}(E)=\begin{cases}
\frac{1}{2}+\frac{1}{2}\log(p-1)-\frac{E^{2}}{2}+\int_{-2}^{2}\frac{1}{2\pi}\sqrt{4-x^{2}}\log\abs{x-E}dx & E<0\\
\frac{1}{2}\log(p-1) & E\geq0
\end{cases}
\]
(N.b. This result will not be used in our arguments in an essential way. We include 
it to clarify the exposition surrounding the following notions.)

With this in hand, we then observe the following important result of Subag--Zeitouni regarding the extremal process
for $H$. For every fixed $N$, if $p$ is even, order the locations of the local minima of $H$
as $x_{\pm1},x_{\pm2},...\in\mathcal{S}^{N}$, where for $x_{i},x_{j}$
two local minima, $|i|<|j|$ if 
\begin{align*}
H(x_{i})\leq & H(x_{j})\,,
\end{align*}
and $x_{i}=-x_{-i}$; if $p$ is odd, order them simply as $x_1,x_2,...\in \mathcal S^N$. Finally, let $m_N$ be the quantity
\[
m_{N}=-E_{0}N+\frac{1}{2\Theta_{p}'(E_{0})}\log N-K_{0}\,,
\]
 where $K_{0}$ is an explicitly defined constant (see \cite[Eq. (2.6)]{SubZeit16}), and $E_0$ is the unique zero of $\Theta_p$.
 (We remark here that $E_0$ is the same constant mentioned in Remark~\ref{rem:explicit-E_0}.)
 
\begin{prop}[{\cite[Theorem 1]{SubZeit16}}]
\label{prop:(Subag-Zeitouni)}  For any $p\geq3$, we have that 
\[
\frac{2}{(3+(-1)^{p})}\sum_{\sigma:\nabla H\restriction_{\sigma}=0}\delta_{H(\sigma)-m_{N}}\xrightarrow[N\to\infty]{(d)}PPP(e^{\Theta'(-E_{0})x}dx)\,,
\]
where $PPP(f(x)dx)$ denotes the Poisson point process of intensity
$f(x)$, and the convergence is in distribution with respect to the
vague topology. 
\end{prop}
In our paper, we do not need the full power of this deep result. Instead we only need 
the following simple corollary of \prettyref{prop:(Subag-Zeitouni)}.
\begin{cor}
\label{cor:loc-of-minima}For any $k\in\N,$ if $x_{1},...,x_{k}\in\cS^{N}$
are the locations of the ground state to the $k$-th smallest local
minima, respectively, we have 
\[
(H_{N}(x_{l})-m_{N})_{l\in[k]}\xrightarrow{(d)}Y\,,
\]
where $Y$ is a random variable supported on all of $\R^{k}$.
\end{cor}

In order to obtain our low-temperature spectral estimates, we will need to control certain natural physical 
quantities, called free energies. Recall that the \emph{free
energy density} corresponding to the partition function $Z_{N}=Z_{N,\beta}$
defined in the introduction, is given by 
\[
F_{N}=\frac{1}{N}\log Z_{N}=\frac{1}{N}\log\int_{\mathcal{S}^{N}}e^{-\beta H(\sigma)}dV(\sigma)\,.
\]
Then, for a Borel set $A\subset\mathcal{S}^{N}$, let 
\begin{align*}
Z_{N}(A)= & \int_{A}e^{-\beta H(\sigma)}dV(\sigma)\,, \qquad \mbox{and} \qquad F_{N}(A)= \frac{1}{N}\log(Z_{N}(A))\,,
\end{align*}
be the \emph{restricted partition function }and \emph{restricted free
energy} of a set $A$, respectively, so that $F_{N}(\mathcal{S}^{N})=F_{N}$.
(This is called the \emph{reduced} \emph{free energy} in \cite{SubGibbs16}.)

The main estimate we use in the low temperature regime is the following result of Subag regarding
the conditional law of the restricted free energy of bands around minima. 
More precisely, for any $x\in\cS^{N}$, $q\in(0,1)$, and any $\epsilon>0$, define the Borel
sets
\begin{align*}
\mbox{Cap}(x,q) & =\{\sigma\in\mathcal{S}^{N}:R(x,\sigma)\geq q\}\,,\\
\mbox{Band}(x,q,\epsilon) & =\{\sigma\in\cS^{N}:R(x,\sigma)\in[q-\epsilon,q+\epsilon]\}\,,
\end{align*}
which are a cap and band respectively around a point $x$ corresponding to an overlap $q$.
These satisfy the following free energy estimates near critical points.

\begin{prop}[{\cite[Proposition 19, Lemma 20]{SubGibbs16}}]
\label{prop:(Subag)-band} For every $p\geq3$, there exists a $\beta_{0}(p)$ and a $0<q_\star (p,\beta)<1$
such that for all $\beta\geq\beta_{0}$, the following holds: 
\begin{enumerate}
\item Let $a_{N}=o(N)$ and $\epsilon_{N}=o(1)$ be two sequences of positive
numbers; then for $J_{N}=(m_{N}-a_{N},m_{N}+a_{N})$ we have for
any $x\in\mathcal{S}^{N}$, $t>0$, 
\[
\lim_{N\to\infty}\sup_{u\in J_{N}}\bigg|\mathbb{P}_{u}\left(\frac{Z_{N}(\mbox{Band}(x,q_*,\epsilon_{N}))}{\E_{u}\left[Z_{N}(\mbox{Band}(x,q_*,\epsilon_{N}))\right]}\leq t\,\right)-\mathbb{P}\left(e^{Y_{*}}\leq t\right)\bigg|=0\,,
\]
for some $Y_{*}$, a normal random variable whose mean and variance
are functions of $p$ alone.
\item Furthermore, there exists $0< q_{\star \star}(p,\beta)<q_\star$ and $\Lambda(p,\beta)>0$ such that for every $x\in\mathcal{S}^{N}$ and every $\eta>0$, 
\begin{equation}
\limsup_{N\to\infty}\sup_{u\in J_{N}}\bigg|\frac{1}{N}\log\left(\E_{u}\left[Z_{N}(\mbox{Band}(x,q_*,\eta N^{-1/2})\right]\right)-\Lambda(p,\beta)\bigg|=0\,,\label{eq:subag-2a}
\end{equation}
and for any fixed $\epsilon>0$, 
\begin{equation}
\limsup_{N\to\infty}\sup_{u\in J_{N}}\frac{1}{N}\log\left(\E_{u}\left[Z_{N}(\mbox{Cap}(x,q_{**})\backslash\mbox{Band}(x,q_*,\epsilon))\right]\right)<\Lambda(p,\beta)\,.\label{eq:subag-2b}
\end{equation}
\end{enumerate}
\end{prop}
\noindent Henceforth, $q_*(p,\beta)$, $q_{**}(p,\beta)$, and $\Lambda(p,\beta)$ will be those constants given by Proposition~\ref{prop:(Subag)-band}. 

\section{Free Energy Estimates}

In this section, we prove the key equilibrium estimate for the proof
of exponentially slow relaxation at low temperature. In particular,
we compute ratios of Gibbs probabilities at the exponential level.
We begin first with the a modification of a classical concentration
estimate. We then turn to the main estimate in the following subsection.
Finally we state as corollaries the precise applications of these
results that we will use in the subsequent sections.

\subsection{Concentration of Restricted Free Energies\label{sub:Concentration-of-RFE}}

We begin by briefly recalling the fact that the restricted free energy
of any Borel set concentrates under both $\mathbb{P}$ and $\mathbb{P}_{u}$.
This estimate is a modification of a classical concentration estimate
for free energies. We include a proof for the reader's convenience.
\begin{lem}
\label{lem:fe-conc} For any Borel set $E\in \cB(\mathcal S^N)$, the restricted free energy corresponding to $H_{N}$,
\[
F_{N}(E)=\frac{1}{N}\log\int_{E}e^{-\beta H_{N}(\sigma)}dV(\sigma),
\]
 concentrates with respect to $\prob$ and, for any $x\in\cS^{N}$,
with respect to the conditional measure $\prob_{u}$. That is, there
is a constant $c>0$ depending only on $\beta$ and $p$ such that
for any $N$ and any $E\in\cB(\cS^{N})$, 

\begin{align*}\prob\left(\abs{F_{N}(E)-\E F_{N}(E)}>\epsilon\right) & \lesssim e^{-cN\epsilon^{2}}\,,\\
\prob_{u}\left(\abs{F_{N}(E)-\E_u F_{N}(E)}>\epsilon\right) & \lesssim e^{-cN\epsilon^{2}}\,.
\end{align*}

\end{lem}
\begin{proof}
Without loss of generality, $V(E)>0$, otherwise $F_N(E)=-\infty$ identically.
Under $\mathbb P$ and, by the equality in distribution in \prettyref{lem:H-reform}, under $\mathbb P_u$, the restricted free energy
is equal in law to 
\[
\frac{1}{N}\log\int_{E}e^{-\beta X(\sigma)}dV(\sigma)\qquad \mbox{where} \qquad
X(\sigma)=\sum_{k=1}^{p}a_{p,k}H_{N,k}(\sigma)+g(\sigma)
\]
for some deterministic, smooth, $g(\sigma)$ and coefficients $a_{p,k}$ suitably chosen depending
on $p$. \color{black}Consider this more general setup and denote this free energy by $F(J,E)$
to make the dependence on the coupling coefficients in $H_{N,k}$
explicit. Observe that 
\[
\frac{\partial}{\partial J_{i_{1},\ldots i_{k}}}X(\sigma)=\frac{a_{p,k}}{N^{(k-1)/2}}\sigma_{i_{1}}\cdots\sigma_{ik}\,,
\]
so that 
\[
\nabla_{J}F(J,E)=\frac{1}{N}\left(-\beta\frac{a_{p,k}}{N^{(k-1)/2}}\left\langle \sigma_{i_{1}}\cdots\sigma_{i_{k}}\right\rangle \right)_{i_1,...,i_k:\, k\leq p}\,,
\]
where $\left\langle \cdot\right\rangle $ denotes integration with
respect to the Gibbs measure induced by $X$ conditioned on the event
$E$. (Since $V(E)>0$ by assumption and $H$ is continuous for each
choice of $J$, $\pi (E) >0$, so this is defined
in the usual sense.) Thus $F$ is $c/\sqrt{N}$- Lipschitz in $J$ for some
$c=c(\beta,a_{p,k})>0$. Since $J$ is a collection of i.i.d.\ Gaussians, this implies
the result by standard Gaussian concentration. 
\end{proof}

\subsection{Refined Free Energy Estimates}

In this subsection, we prove the main estimate we need regarding $\pi$
at low temperature. As is often the case, this result reduces to showing
that certain free energy differences are negative. These results will
come from combining the estimates from \prettyref{sub:Previous-results}
with the concentration estimate from \prettyref{sub:Concentration-of-RFE}.
The goal of this subsection is to prove the following proposition.
Recall the notation $x_{\pm1},\cdots\in\cS^{N}$ regarding the lowest
critical points from the end of \prettyref{sub:Previous-results}.
\begin{prop}
\label{prop:cap-free-energy}Fix any $k$ and let $x_{*}=x_{k}$.
Fix any $\eta>0$ and let $A(x)=\mbox{Band}(x,q_{*},\eta)$ and $B(x)=\mbox{Cap}(x,q_{**})\backslash A(x)$,
where the sets $\mbox{Band}$ and $\mbox{Cap}$ were defined in \prettyref{sub:Notation}.
Then there exists a $\beta_{0}(p)$ such that for every $\beta\geq\beta_{0}$,
there exists $c(\beta)>0$ such that, 
\[
\lim_{N\to\infty}\mathbb{P}\left(F_{N}(B(x_{*}))-F_{N}(A(x_{*}))<-c\right)=1\,.
\]
\end{prop}
Before proving this proposition we will need estimates on $F_{N}(A(x))$
and $F_{N}(B(x))$ under $\prob_{u}$. To this end, begin by observing
that by \prettyref{lem:fe-conc}, for any $x\in\mathcal{S}^{N}$,
$F_{N}(B(x))$ and $F_{N}(A(x))$ concentrate around their respective
means; in particular, there exists a constant $c(\beta,p)>0$ such that for
every $\delta>0$, 
\begin{equation}
\prob_{u}(|F_{N}(B(x))-\E_{u}F_{N}(B(x))|>\delta)\lesssim e^{-cN\delta^{2}}\,,\label{eq:fe-concentration}
\end{equation}
and similarly for $F_{N}(A)$. 

We begin the proof with the following two lemmas. Recall the definitions
of $q_{*},E_{0}$, and $\Lambda$ from \prettyref{sub:Previous-results}.
The first lemma shows that the probability that $A(x)$ has free energy that is smaller than $\Lambda$ is vanishing in the limit.
\begin{lem}
\label{lem:free-energy-A}Let $a_{N}$ and $J_{N}$ be as in \prettyref{prop:(Subag)-band}
and $\eta$ and $A(x)$ be as in \prettyref{prop:cap-free-energy}.
There exists $c(\beta,p)>0$ such that for every $\delta>0$, 
\begin{equation}
\sup_{x\in\cS^{N}}\sup_{u\in J_{N}}\prob_{u}\left(F_{N}(A(x))<\Lambda(p,\beta)-\delta\right)\lesssim e^{-cN\delta^2}\,.\label{eq:fe-A-bound}
\end{equation}
\end{lem}
\begin{proof}
Fix $x\in\cS^{N}$ and define $\tilde{A}(x)=\mbox{Band}(x,q_{*},\eta N^{-1/2})$.
Observe that because $\tilde{A}\subset A$, we have $F(\tilde{A})\leq F(A)$,
from which it follows that

\begin{equation}
\limsup_{N\to\infty}F_{N}(A(x))-\Lambda(p,\beta)\geq\limsup_{N\to\infty}F_{N}(\tilde{A}(x))-\Lambda(p,\beta)\,.\label{eq:fe-fe-tilde}
\end{equation}
Define the set 
\[
V=\left[\frac{1}{N}\log\E_{u}[Z_{N}(\tilde{A}(x))]-\frac{K}{N},\frac{1}{N}\log\mathbb{\E}_{u}[Z_{N}(\tilde{A}(x))]+\frac{K}{N}\right].
\]
By item (1) of \prettyref{prop:(Subag)-band}, with the choice $\epsilon_{N}=\eta N^{-1/2}$,
and the Gaussian tails of $Y_{*}$ (defined there), there is an absolute
constant, $c>0$ such that for any $K$ sufficiently large,

\begin{equation}
\sup_{u\in J_{N}}\mathbb{\prob}_{u}\left(F_{N}(\tilde{A}(x))\in V^{c}\right)\leq\exp(-cK^{2})+o(1)\,.\label{eq:log-of-expectation}
\end{equation}
With these results in hand, observe that

\begin{align*}
\E_{u}[F_{N}(\tilde{A}(x))]& = \E_{u}[F_{N}(\tilde{A}(x))\boldsymbol{1}\{F_{N}(\tilde{A}(x))\in V\}]+\E_{u}[F_{N}(\tilde{A}(x))\boldsymbol{1}\{F_{N}(\tilde{A}(x))\in V^{c}\}]\,.
\end{align*}
Combining this with \prettyref{eq:log-of-expectation} and the Cauchy-Schwarz
inequality, we have that for each fixed $K$ large enough, for every
$u\in J_{N}$, 
\begin{align*}
\left|\E_{u}[F_{N}(\tilde{A}(x))]-\frac{1}{N}\log\mathbb{\E}_{u}[Z_{N,\beta}(\tilde{A}(x))]\right| & \leq\left(\E_{u}\left[\left(F_{N}(\tilde{A}(x))\right)^{2}\right]\right)^{\frac{1}{2}}\left(\exp(-cK^{2})+o(1)\right)^{\frac{1}{2}}+o(1).
\end{align*}

We estimate the right hand side as follows. Splitting up the expectation, and using Eq.~\eqref{eq:fe-concentration}, we obtain for every $u\in J_N$,
\begin{align*}
\E_{u}\left[F_{N}(\tilde{A}(x))^{2}\right]\leq & \E_{u}\left[F_{N}(\tilde{A}(x))^{2}\left(\boldsymbol{1}\{F_{N}(\tilde{A}(x))\leq\E_{u}F_{N}(\tilde{A}(x))\}+\boldsymbol{1}\{F_{N}(\tilde{A}(x))>\E_{u}F_{N}(\tilde{A}(x))\}\right)\right]\, \\
\leq & \left(\E_{u}[F_{N}(\tilde{A}(x))]\right)^{2}+\sup_{u\in J_{N}}\int_{0}^{\infty} 2(\mathbb E_u F_N(\tilde A(x)) + t) \cdot e^{-cNt^{2}}dt\,.
\end{align*}
We now bound $\E_{u} [F_{N}(\tilde{A}(x))]$ uniformly in $u\in J_{N}$:
letting $\mbox{Var}_{u}$ denote the variance with respect to ${\mathbb{P}}_{u}\color{black}$,
we have that 
\[
\E_{u}[F_{N}(\tilde{A}(x))]\leq\frac{1}{N}\log\int_{\mathcal{S}^{N}}\E_{u}[e^{-\beta H(\sigma)}]dV(\sigma)\leq\frac{\beta u}{N}+\frac{1}{N}\sup_{\sigma\in\mathcal{S}^{N}}\frac{\beta^{2}}{2}\mbox{Var}_{u}(H(\sigma))\,,
\]
where we use Jensen inequality for the first inequality, and the $\mathbb{P}_{u}$
conditional distribution of $H(\sigma)$ given by \prettyref{lem:H-reform}
for the second.

Recall now, from the definition of $J_{N}$, that $\frac{u}{N}$ is
bounded by some constant that depends only on $p$. Combining the
above with the covariance bound obtained in \prettyref{lem:H-reform}
(independent of $u$) to bound $\mbox{Var}_u (H(\sigma)) \lesssim_p N$, we see that 
\begin{align*}
\sup_{u\in J_{N}}\left(\E_{u}\left[\left(F_{N}(\tilde{A}(x))\right)^{2}\right]\right)^{\frac{1}{2}} & \lesssim_{p,\beta}1+o(1)\,.
\end{align*}

Altogether, we see that 
\begin{align}
\sup_{u\in J_{N}}\abs{\E_{u}[F_{N}(\tilde{A}(x))]-\frac{1}{N}\log\E_{u}[Z_{N}(\tilde{A}(x))]}& \lesssim_{p,\beta} e^{-cK^{2}}+o(1)\,.\label{eq:switch-log-expectation}
\end{align}
Although the $o(1)$ term is not uniform in $K$, for any $\delta>0$,
there exists a $K$ such that for $N$ sufficiently large the above
difference is less than $\delta/6.$ Moreover, by item (2) of \prettyref{prop:(Subag)-band},
\begin{align}
\lim_{N\to\infty}\sup_{u\in J_{N}}|\frac{1}{N}\log\E_{u}[Z_{N}(\tilde{A}(x))]-\Lambda(p,\beta)|= & 0\,,\label{eq:fe-like-lambda}
\end{align}
so for any $\delta>0$, for $N$ sufficiently large the difference
in \prettyref{eq:fe-like-lambda} is less than $\delta/6$. By \prettyref{eq:switch-log-expectation}
and the finiteness of $\Lambda(p,\beta)$, for every $\delta>0$,
there exists $K$ large enough that for all $N$ sufficiently large,
\begin{align*}
\sup_{u\in J_{N}}|\E_{u}[F_{N}(\tilde{A}(x))]-\Lambda(p,\beta)|\leq & \sup_{u\in J_{N}}|\mathbb{\E}_{u}[F_{N}(\tilde{A}(x))]-\frac{1}{N}\log\mathbb{\E}_{u}[Z_{N}(\tilde{A}(x))]|\,.\\
 & +\sup_{u\in J_{N}}|\frac{1}{N}\log\mathbb{\E}_{u}[Z_{N,\beta}(\tilde{A}(x))]-\Lambda(p,\beta)|\\
< & \delta/3\,.
\end{align*}
Then by the triangle inequality and \prettyref{lem:fe-conc}, for
all such $N$, 
\begin{align*}
\sup_{u\in J_{N}}\mathbb{\prob}_{u}(|F_{N}(\tilde{A}(x))-\Lambda(p,\beta)|>\delta)\leq & \sup_{u\in J_{N}}\prob_{u}(|F_{N}(\tilde{A}(x))-\E_{u}[F_{N}(\tilde{A}(x))]|>\delta/3)\\
\lesssim & e^{-cN\delta^{2}/9}\,.
\end{align*}
Combined with \prettyref{eq:fe-fe-tilde}, and the observation that
every estimate in this proof has been independent of $x\in\mathcal{S}^{N}$,
we obtain for every $\delta>0$,
\[
\sup_{x\in\mathcal{S}^{N}}\sup_{u\in J_{N}}\prob_{u}\left(F_{N}(A(x))<\Lambda(p,\beta)-\delta\right)\lesssim e^{-cN \delta^2/9}\,.\qedhere
\]

\end{proof}
Now that we know that $F_N(A)$ is large with high probability, we want
the corresponding estimate to show that the probability that $F_N(B)$
is larger than $\Lambda-\delta$ (for $\delta$ small enough) is small.
\begin{lem}
\label{lem:free-energy-B}Let $a_{N}$ and $J_{N}$ be as in \prettyref{prop:(Subag)-band}
and $\eta$ and $B(x)$ be as in \prettyref{prop:cap-free-energy}.
There exists $c(\beta,p)>0$ such that for every $\delta>0$ sufficiently small, 
\[
\sup_{x\in\cS^{N}}\sup_{u\in J_{N}}\prob_{u}\left(F_{N}(B(x))>\Lambda(p,\beta)-\delta\right)\lesssim e^{-cN\delta^2}\,.
\]
\end{lem}
\begin{proof}
For any $x\in\mathcal{S}^{N}$. By 
Jensen's inequality, and item (2) of \prettyref{prop:(Subag)-band}
(combined with the rotational invariance of $H$ which implies that
the estimate is uniform over $\mathcal{S}^{N}$), there exists a $\delta>0$
such that, 
\begin{align}
\limsup_{N\to\infty}\sup_{x\in\mathcal{S}^{N}}\sup_{u\in J_{N}}\E_{u}[F_{N}(B(x))] & =\limsup_{N\to\infty}\sup_{u\in J_{N}}\E_{u}[F_{N}(B(x))]\nonumber \\
 & \leq\limsup_{N\to\infty}\sup_{u\in J_{N}}\frac{1}{N}\log\E_{u}[Z_{N,\beta}(B(x))]\nonumber \\
 & \leq \Lambda(p,\beta)-3\delta\,.\label{eq:fe-b-lambda}
\end{align}
Thus for some sufficiently large $N$, the left hand side is less
than $\Lambda(p,\beta)-2\delta$. Combined with the concentration
of the free energy under $\mathbb{P}_{u}$ given by \prettyref{eq:fe-concentration},
we see that there exists a constant $c(\beta,p)>0$ such that for sufficiently large $N$, 
\begin{align}
\sup_{x\in\mathcal{S}^{N}}\sup_{u\in J_{N}}\prob_{u}\left(F_{N}(B(x))>\Lambda(p,\beta)-\delta\right)\leq & \sup_{x\in\mathcal{S}^{N}}\sup_{u\in J_{N}}\mathbb{\prob}_{u}\left(|F_{N}(B(x))-{\mathbb{E}}_{u}F_{N}(B(x))|>\delta\right)\nonumber \\
\lesssim & e^{-cN\delta^{2}}\,,\label{eq:fe-b}
\end{align}
as desired.
\end{proof}

In order to complete the proof of \prettyref{prop:cap-free-energy}, it remains to move from free energy differences under $\mathbb P_u$ for $u\in J_N$ to free energies under $\mathbb P$ around the (random) point $x_k$. 

For this, we 
will need the following result which is a standard application of the Kac--Rice formula combined with \prettyref{prop:(Subag-Zeitouni)} (see \cite[Lemma 38]{SubGibbs16}).
Recall that in order to apply the Kac--Rice formula, one needs some basic smoothness
criteria, called \emph{tameness}. More precisely, a random field $G$ is \emph{tame} 
if it satisfies criteria, (a)--(g) in Theorem 12.1.1 of~\cite{AdlerTaylor} and
the random field $(H(x),G(x))_x$ is stationary random field. In~\cite{SubGibbs16}, this was applied in the case where $G(x)$ is the restricted free energy of some set around $x$, using that such free energies are tame.

\begin{lem}[Lemma 38 of~\cite{SubGibbs16}]\label{lem:kac-rice-application}
Let $G$ be tame, let $J_N = (m_N -a_N, m_N +a_N)$ for $a_N =o(N)$ and define $\mathscr C(J_N)=\{\sigma:\nabla H\restriction_\sigma =0, H(\sigma) \in J_N\}$. If $D_N$ is an interval, there exist constants $C,c_p>0$ given by~\emph{\cite[Eq.~(2.8)]{SubGibbs16}} such that for every $x\in \mathcal S^N$,
\begin{align*}
\mathbb E\Big[ \sum_{\sigma \in \mathscr C(J_N)}\boldsymbol 1\{G(\sigma)\in D_N\}\Big] \leq C\int_{J_N} e^{c_p(u-m_N)} \big[\mathbb P_u (G(x))\big]^{1/2} du\,.
\end{align*}
\end{lem}

\begin{proof}[{\textbf{\emph{Proof of \prettyref{prop:cap-free-energy}}}}]
\textbf{\emph{ }}For each $x\in\cS^{N}$, $\delta>0$, define the event
\[
E(x,\delta)=\{F_{N}(B(x))-F_{N}(A(x))\leq -\delta\}\,.
\]
We begin by finding a $\delta>0$ for which
\begin{equation}
\sup_{x\in \mathcal S^N} \sup_{u\in J_{N}}\prob_{u}\left(E^c(x,\delta)\right)\lesssim e^{-cN\delta^2}\,,\label{eq:want-to-show}
\end{equation}
for some $c(\beta,p)>0$. 
With this goal in mind, observe that for any $\delta>0$, a union
bound gives 
\begin{align*}
\sup_{x\in\mathcal{S}^{N}}\sup_{u\in J_{N}}\prob_{u}(F_{N}(A(x))-F_{N}(B(x))<\delta)\leq & \sup_{x\in\mathcal{S}^{N}}\sup_{u\in J_{N}}\prob_{u}(F_{N}(A(x))<\Lambda(p,\beta)-\delta)\\
 & +\sup_{x\in\mathcal{S}^{N}}\sup_{u\in J_{N}}\prob_{u}(F_{N}(B(x))>\Lambda(p,\beta)-2\delta)
\end{align*}
whence using the sufficiently small $\delta>0$ given by \prettyref{lem:free-energy-B},
combining Lemmas~\ref{lem:free-energy-A}--\ref{lem:free-energy-B} with the definition of $E(x,\delta)$ yields the desired~\eqref{eq:want-to-show}.
In order to conclude the proof, we recall that $J_N = (m_N - a_N, m_N + a_N)$ for $a_N =o(N)$ and $\mathscr C(J_N)=\{\sigma:\nabla H\restriction_\sigma =0, H(\sigma)\in J_N\}$. By Markov's inequality, we bound the quantity,
\begin{align*}
\mathbb P(\exists \sigma \in \mathscr C(J_N): E^c(\sigma,\delta) \mbox{ holds}) \leq \mathbb E \Big[\sum_{\sigma\in \mathscr C(J_N)} \boldsymbol1\{E^c(\sigma,\delta)\}\Big]\,.
\end{align*}
Taking $G(x)= F_N(B(x))-F_N(A(x))$ and $D_N = (-\delta,\infty)$ in~\prettyref{lem:kac-rice-application}, and noting that these restricted free energies are tame, so that $G$ is tame, 
we obtain for $c(\beta,p)>0$,
\begin{align*}
\mathbb P(\exists \sigma \in \mathscr C(J_N): E^c (\sigma,\delta) \mbox{ holds}) & \leq 2C e^{c_p a_N}\sup_{u \in J_N} \sqrt{\mathbb P_u[E^c(x,\delta)]} \\
& \lesssim e^{c_pa_N - cN\delta^2/2}\,,
\end{align*}
which is exponentially small in $N$ since $a_N =o(N)$. Specifically, for any fixed $k$, we have $\mathbb P(x_k \in \mathscr C(J_N), E^c(x_k,\delta))=o(1)$ while by \prettyref{cor:loc-of-minima}, for any fixed $k$, 
\[\lim_{N\to\infty} \mathbb P(x_k \in \mathscr C(J_N))= \lim_{N\to\infty} \mathbb P(H(x_k) \in J_N) =1\,,
\]
so that by a union bound, 
\[\mathbb P(E^c(x_*,\delta)) \leq \mathbb P(x_* \notin \mathscr C(J_N)) +\mathbb P(x_* \in \mathscr C(J_N), E^c(x_*,\delta))=o(1)\,. \qedhere 
\]
\end{proof}

\subsection{Ratios of Gibbs Weights Near Local Minima}

Now that we have the free energy control from \prettyref{prop:cap-free-energy},
we can control ratios of certain Gibbs probabilities. The corollaries
capture the specific application of these estimates that we will need
in the subsequent. 

Recall the notation $x_{\pm1},\ldots$, regarding the lowest critical
points from \prettyref{sub:Previous-results}. Following this convention,
for any $x_{k}$, we define the following subsets of $\mathcal{S}^{N}$: \color{black}
\begin{align}
A_{k} & = \mbox{Cap}(x_{k},q_{**}+\epsilon N^{-1/2})=\{\sigma\in\mathcal{S}^{N}:R(\sigma,x_{k})>q_{**}+\epsilon N^{-1/2}\}\nonumber \\
B_{k}& =\mbox{Band}(x_{k},q_{**}+\tfrac{\epsilon}{2}N^{-1/2},\tfrac{\epsilon}{2}N^{-1/2})=\{\sigma\in\mathcal{S}^{N}:R\left(\sigma,x_{k}\right)\in[q_{**},q_{**}+\epsilon N^{-1/2}]\} \nonumber\\
B_{k}^{*}& = \mbox{Band}(x_{k},q_{\ast},\epsilon)=\{\sigma\in\mathcal{S}^{N}:R\left(\sigma,x_{k}\right)\in[q_{\ast}-\epsilon,q_{*}+\epsilon]\} \label{eq:A-B-def} 
\end{align}
for some sufficiently small $\epsilon=O(1)$ chosen such that $2\epsilon<q_{*}-q_{**}$
(such a choice of $\epsilon>0$ exists since $q_{*}>q_{**}$). The first estimate shows that the ratio
of Gibbs probabilities is exponential in $N$. 
\begin{cor}
\label{cor:sets-decay}For every $p\geq3$, there exists some $\beta_{0}(p)$
such that for all $\beta>\beta_{0}$, there exist $c_{1}(p,\beta),c_{2}(p,\beta)>0$
such that for any fixed $k$, with $\mathbb{P}$-probability going
to $1$ as $N\to\infty$, 
\begin{align}
\pi(B_{k})\pi(A_{k})^{-1}\leq & c_{1}\exp(-c_{2}N)\,,\label{eq:cap-to-band-ratio}
\end{align}
and with $\mathbb P$-probability going to $1$ as $N\to\infty$, $\pi(B_{k})\leq c_{1}\exp(-c_{2}N)$.\end{cor}
\begin{proof}
The estimate is a direct consequence of \prettyref{prop:cap-free-energy}
for the corresponding choice of $k$ and the choice $\eta=\epsilon$.
To see this, first observe that $H$ is smooth so that $\pi$ is absolutely
continuous with respect to $dV$ and we do not need to worry about
the mass of the boundaries of the sets $A_{k},B_{k}$. 
Observe that the above ratio can be
understood as a free energy difference:
\begin{align*}
\frac{1}{N}\log(\pi(B_{k})\pi(A_{k})^{-1})= & \frac{1}{N}\log\left(\frac{Z_{N}(B_{k})}{Z_{N}(A_{k})}\right)=F_{N}(B_{k})-F_{N}(A_{k})\,.
\end{align*}

Since $B_{k}^{*}\subset A_{k}$, we have that $\pi(A_{k})\geq\pi(B_{k}^{*})$;
moreover, since $B_{k}\subset \mbox{Cap}(x_k,q_{**})\backslash B_{k}^{*}$,
we have that $\pi(B_{k})\leq\pi(\mbox{Cap}(x_k,q_{**})\backslash B_{k}^{*})$.
With these observations in hand, we see that \prettyref{prop:cap-free-energy}
implies that for all $\beta\geq\beta_{0}$, where $\beta_{0}(p)$
is given by \prettyref{prop:cap-free-energy}, 
\begin{align*}
\pi(B_{k})\pi(A_{k})^{-1}\leq & \pi(\mbox{Cap}(x_k,q_{**})\backslash B_{k}^{*})\pi(B_{k}^{*})^{-1}\\
\leq & \exp\left[N(F_{N}(\mbox{Cap}(x_k,q_{**})\backslash B_{k}^{*})-F_{N}(B_{k}^{*}))\right]\\
\leq & c_{1}\exp(-c_{2}N)\,,
\end{align*}
for some $c_1(p,\beta), c_2(p,\beta)>0$ with $\mathbb{P}$-probability going to $1$ as $N\to\infty$.\end{proof}
\begin{cor}
\label{cor:pigeonhole} For $k\in\{1,2,3\}$, the sets $A_{k},B_{k}$
defined in \prettyref{eq:A-B-def}, satisfy 
\begin{align*}
\lim_{N\to\infty}\prob(\exists k\in\{1,2,3\}:\pi((A_{k}\cup B_{k})^{c})\geq & \frac{1}{2})=1.
\end{align*}
\end{cor}
\begin{proof}
Each element of $\{x_{k}\}_{i\in\{1,2,3\}}$ has a corresponding $\mbox{Cap}(x_{k},q_{**})$
and with probability going to $1$ as $N\to\infty$, all three of $H(x_i)\in J_N$ so that the three caps
are disjoint by the choice of $q_{**},\epsilon$ and \cite[Cor. 13]{SubGibbs16}.
Then all three $x_{k}$'s having $\pi(\mbox{Cap}(x_{k},q_{**}))\geq\frac{1}{2}$
would contradict $\pi(\mathcal{S}^{N})=1$.
\end{proof}

\section{Spectral Gap Inequalities for Gibbs Measures\label{sec:Spectral-Gap-inequalities}}

Before turning to the proofs of the main results, we take a brief
pause from the above probabilistic considerations and turn to the
main analytical tools. Some of the results from this section are classical.
We restate them for the completeness. We also prove an adaptation
to our setting of a standard bound on the spectral gap. 

The setting of this section is more general than that of other sections.
Let $M$ be a smooth compact boundaryless Riemannian manifold with
metric $g$ and normalized volume measure $dV$. Let $U\in C^{\infty}(M)$.
As before, we define the \emph{Gibbs} measure, $\pi$, by
\[
d\pi(x)=\frac{e^{-\beta U(x)}}{Z}dV(x)
\]
and the associated operator $\mathcal{L}=\frac{1}{2}\Delta-\frac{\beta}{2}g(\nabla U,\nabla)$
with domain $C^{\infty}(M)\subset L^{2}(M)$ where $\nabla=\nabla_{g}$
is again the covariant derivative and $\Delta$ is the corresponding
Laplacian. As $-\cL$ is a uniformly elliptic operator with smooth
and bounded coefficients, its eigenfunctions are $C^{\infty}$ \cite{GilTrud01,EvansPDE}.
Thus by symmetry of $-\cL$ on $C^{\infty}(M)$ with respect to $\pi$,
it is essentially self-adjoint there \cite{LaxFunctional02}. Furthermore, it's domain, $H^1(\pi)$,
is a compact subset of $L^2(\pi)$ so that it has pure point spectrum which we denote by $0=\lambda_{0}\leq\lambda_{1}\leq\ldots$
In particular, it has Markov semi-group $P_{t}=e^{t\cL}$ .

We say that a measure $\mu$ on $M$ satisfies a \emph{Poincar\'e inequality}
with constant $C>0$ if for every $f\in C^{\infty}(M)$, 
\[
\mbox{Var}_{\mu}(f)\leq C\int_{M}g(\nabla f,\nabla f)d\mu\,.
\]
We say a measure $\mu$ satisfies a \emph{log-Sobolev inequality}
with constant $c>0$ if for every $f\in C^{\infty}(M)$, 
\begin{align}\label{eq:log-sobolev-def}
\int_{M}f^{2}\log\left(\frac{f^{2}}{\int_{M}f^{2}d\mu}\right)d\mu\leq2c\int_{M}g(\nabla f,\nabla f)d\mu.
\end{align}

Corresponding to $\pi_{N},$ we define the \emph{Dirichlet form} by
\begin{equation}
\cE(f,h)=\int_{\mathcal{S}^{N}}g(\nabla f,\nabla h)d\pi_{N}\,.\label{eq:dirichlet-form}
\end{equation}
By the Courant-Fischer min-max principle \cite{LaxFunctional02},
the spectral gap $\lambda_{1}$ of the operator $-\mathcal{L}$ is
given by the variational formula 
\begin{align}
\lambda_{1}= & \min_{{f\in C^{\infty},\,
\|\nabla f\|_{L_{\pi}^{2}}\neq0
}
}\frac{\mathcal{E}(f,f)}{\mbox{{Var}}_{\pi}f}\label{eq:variational-form-gap}
\end{align}
As a result, observe that if $\pi$ satisfies a Poincar\'e inequality
with constant $C>0$ then the spectral gap of the corresponding operator,
$-\mathcal{L}$, has $\lambda_{1}\geq\frac{1}{C}$. We also remind
the reader of the following classical fact.
\begin{lem}
\label{lem:LS-implies-PI}If $\mu$ satisfies a log-Sobolev inequality
with constant $c>0$, then it also satisfies a Poincar\'e inequality
with constant $c$. 
\end{lem}
The proof of this result as well as the following two results is very
classical and can be seen, for example, in \cite{AGZ10,GuiZeg03}.
There are many ways to verify that a Gibbs measure satisfies these
inequalities. The two that we will be using are the following classical
estimates. The first is a stability estimate for Poincar\'e inequalities.
\begin{prop}
\label{prop:poincare-stability-1}(Stability of Poincar\'e Inequalities)
Let $M$ be a Riemannian manifold and suppose that $d\nu=\frac{e^{-U}}{Z}d\mu$
where $Z=\int_{M}e^{-U}d\mu$, $\mu,\nu$ are two probability measures
on $M$, and $U\in C_{b}(M)$. Then if $\mu$ satisfies the Poincar\'e
inequality with constant $C>0$, then $\nu$ satisfies the Poincar\'e
inequality with constant $Ce^{2\beta\left(\max U-\min U\right)}$.
\end{prop}
The next result is one of the foundational results regarding to Bakry
and Emery's curvature dimension. 
\begin{prop}
\label{prop:(Curvature-Energy-Balance-Lemma)}(Curvature-Energy Balance
theorem, Bakry-Emery) Let $(M,\pi)$ be a Riemannian manifold with
metric tensor $g$ and Gibbs measure $\pi$ corresponding to energy
$U$. Let $Ric$ denote the Ricci tensor on $M$ and let $Hess$ denote
the covariant Hessian operator. If there exists a $c>0$ such that
at every point $\sigma$ in $M$ and every $v\in T_{\sigma}M$, the
inequality 
\[
Ric(v,v)+Hess(U)(v,v)\geq cg(v,v)
\]
holds, then $\pi$ admits a log-Sobolev inequality with constant $c$.
\end{prop}
Before stating the final result of this section, we make the following
definitions. For any Borel set $A,$ define the $\epsilon$-enlargement
of $A$ by 
\[
A_{\epsilon}=\{x:d(x,A)\leq\epsilon\}\,,
\]
where $d(x,A)=\inf_{y\in A}d(x,y)$, and for any $y\in M$, let $B_{\epsilon}(y)=\{x:d(x,y)\leq\epsilon\}$.
We now turn to showing a conductance-type
upper bound for the spectral gap, which is a standard adaptation of a canonical conductance bound for Markov processes to our setup.
\begin{prop}
\label{prop:conductance-bound} (Conductance bound) Let $x\in M$
be a point with injectivity radius $R>0$. Suppose that there is an
$0<r<R$ such that $\pi(B_{r}(x))>0$ and let $A=B_{r}(x)$. Then
for all $\epsilon>0$ sufficiently small (i.e., $r+\epsilon<R$),
$\pi(A_{\epsilon}^{c})>0$, and $\pi(A)\pi(A_{\epsilon}^{c})>4\pi(A_{\epsilon}\backslash A)$.
Then, 
\[
\lambda_{1}\leq\frac{9\epsilon^{-2}\pi(A_{\epsilon}\backslash A)}{\pi(A)\pi(A_{\epsilon}^{c})-4\pi(A_{\epsilon}\backslash A)}\,.
\]
\end{prop}
\begin{rem}
Observe that this estimate cannot be sharp as its asymptotic order
in $\epsilon$ is $O(\epsilon^{-1})$ (under certain conditions on
$U$ and $M$). See for example \cite{BakLed96,Led94}.\end{rem}
\begin{proof}
Fix any $x\in M$ and an $\epsilon$ and $r$ satisfying the above
conditions and let $B=A_{\epsilon}\backslash A\supset\partial A$.
Consider the following test function: 
\begin{align*}
f(\sigma)= & \begin{cases}
\pi(A) & \mbox{{on}}\,\,(A_{\epsilon})^{c}\\
-\pi(A^{c}) & \mbox{{on}}\,\, A\\
-\pi(A^{c})+\eta(\epsilon^{-1}d(\sigma,A)) & \mbox{{else}}
\end{cases}
\end{align*}
where $\eta\in C^{\infty}([0,1])$ and satisfies $\eta(0)=0,\eta(1)=1$
and $\sup_{[0,1]}|\frac{d\eta}{dx}|\leq3$ . For concreteness, we
use the function 
\begin{align*}
\eta(x)= & \begin{cases}
\exp(1-\frac{1}{1-(x-1)^{2}}) & \mbox{for }x\in(0,1]\\
0 & \mbox{at }x=0
\end{cases}\,,
\end{align*}
so that certainly, $\sup_{x\in[0,1]}|\frac{d\eta}{dx}(x)|\leq 3$.

First note that $f$ is trivially smooth on $A_{\epsilon}^{c}$ because
it is constant. Since $r+\epsilon$ is less than the injectivity radius,
it is canonical that $d(x,A)$ is smooth in $B_{r+\epsilon}(x)$.
By composition of $\eta$ with $d$ we see that $f\in C^{\infty}(M)$;
moreover, it satisfies the gradient estimate $\sup_{\sigma\in\mathcal{S}^{N}}g(\nabla f,\nabla f)\leq9\cdot\epsilon^{-2}$,
and for $x\in B^{c}$ we have that $\nabla f\equiv0$. By assumption,
$\pi(B)>0$, and on $B\backslash\partial B$, $g(\nabla f,\nabla f)>0$
so that $\|\nabla f\|_{L_{\pi}^{2}}\neq0$. 
Together, this implies that 
\begin{align*}
\mathcal{E}(f,f)= & \int_{M}g(\nabla f,\nabla f)d\pi=\int_{B}g(\nabla f,\nabla f)d\pi\leq9\cdot\epsilon^{-2}\cdot\pi(B)\,.
\end{align*}
At the same time, 
\begin{align*}
|\int_{M}f(\sigma)d\pi(\sigma)|\leq & |\int_{M\backslash B}fd\pi|+|\int_{M}fd\pi|\leq2\pi(B)\,,
\end{align*}
and moreover, 
\begin{align*}
\int_{M}f(\sigma)^{2}d\pi(\sigma)\geq & \int_{M\backslash B}f(\sigma)^{2}d\pi(\sigma)\geq\pi(A)\pi(A^{c})-\pi(A)^{2}\pi(B)\,.
\end{align*}
Therefore, 
\begin{align*}
\mbox{Var}_{\pi}f= & \int_{M}f(\sigma)^{2}d\pi(\sigma)-\left(\int_{M}f(\sigma)d\pi(\sigma)\right)^{2}\\
\geq & \pi(A)\pi(A^{c})-\pi(A)^{2}\pi(B)-4\pi(B)^{2}\,.
\end{align*}
Then, substituting $B=A_{\epsilon}\backslash A$, 
\begin{align*}
\pi(A)\pi(A^{c})-\pi(A)^{2}\pi(B)-4\pi(B)^{2} & \geq\pi(A)\pi(A_{\epsilon}^{c})-4\pi(A_{\epsilon}\backslash A)
\end{align*}
Plugging in this choice of $f$ as a test function in \prettyref{eq:variational-form-gap},
and using the upper bound on the Dirichlet form and lower bound on
the variance, we see the desired bound on $\lambda_1$.
\end{proof}

\section{Proof of Main Theorem}

In this section we prove the lower bound
for the relaxation time (inverse of the spectral gap) of the Langevin
dynamics of the of the spherical $p$-spin model at low temperatures,
using the estimate on the free energy ratio obtained in \prettyref{prop:cap-free-energy}
along with the conductance bound of the previous section. We also
prove a matching (exponential in $N$) upper bound on the relaxation
time which holds at all temperatures and prove that a much stronger
$O(1)$ upper bound, along with a log-Sobolev inequality, holds at high temperatures as expected.

\subsection{Low Temperature}

At sufficiently low temperatures we prove matching (up to constants)
upper and lower bounds on~$\lambda_{1}$. We begin with the lower bound.
Recall first the following classical fact which can be seen by an explicit calculation (see, e.g., \cite{Chav84}). 
\begin{fact}
\label{fact:sphere-eigenvalue} The spectral gap of $-\Delta$ on
$\mbox{\ensuremath{\mathcal{S}}}^{N}=S^{N-1}(\sqrt{N})$ is given
by 
\[
\lambda_{1}=1-\frac{1}{N}\,,
\]
and has eigenspace with multiplicity $N$. Furthermore, the Ricci
tensor everywhere satisfies 
\[
Ric=(1-\frac{1}{N})g\,.
\]
\end{fact}
The above allow us to obtain the following lower bound on the gap
of $-\mathcal{L}$ at all $\beta>0$:
\begin{lem}
\label{lem:lower-bound} For every $\beta>0$, and all $p\geq3$,
there exists a $c(p)>0$ such that the Langevin dynamics of the spherical
$p$-spin model has, 
\[
\lim_{N\to\infty}\mathbb{P}(\lambda_{1}\geq\exp(-c\beta N))=1\,.
\]
\end{lem}
\begin{proof}
Since the Laplacian on $\mathcal{S}^{N}$ has spectral gap $1-o(1)$
(see \prettyref{fact:sphere-eigenvalue}), it follows from the variational
form of the gap, \prettyref{eq:variational-form-gap}, that $d\mu=dV$
on $\mathcal{S}^{N}$ satisfies the Poincar\'e inequality with constant
$1-o(1)$. By \prettyref{lem:(Auffinger-Ben-Arous)}, and the stability
of the Poincar\'e Inequality under Gibbsian perturbations (taking
$M=\mathcal{S}^{N}$ and $\nu=\pi,d\mu=dV$ in \prettyref{prop:poincare-stability-1}),
there exists a $c>0$ such that $\pi$ satisfies the Poincar\'e inequality
with constant 
\begin{align*}
C_{*}= & (1-o(1))\exp(4c\beta N)\,,
\end{align*}
with $\mathbb{P}$-probability tending to $1$ as $N\to\infty$. We
deduce that 
\[\lim_{N\to\infty}\mathbb{P}\left(\lambda_{1}\geq\frac{1}{2}\exp(-4c\beta N)\right)=  1\,. \qedhere
\]\end{proof}
We now prove the following upper bound on the eigenvalue gap.
\begin{lem}
\label{lem:Upper-bound}For every $p\geq3$, there exists a $\beta_{0}(p)>0$
such that for all $\beta\geq\beta_{0}$, there exist $c_{1}(p,\beta),c_{2}(p,\beta)>0$
such that the Langevin dynamics for the spherical $p$-spin model
on $\mathcal{S}^{N}$ satisfies, 
\begin{align*}
\lim_{N\to\infty}\mathbb{P}\left(\lambda_{1}\leq c_{1}\exp(-c_{2}N)\right)= & 1\,.
\end{align*}
\end{lem}
\begin{proof}
For every $N$, every realization of the disorder $\{J_{i_{1},...,i_{p}}\}_{\{i_{1},...,i_{p}\}\subset[N]}$,
choose the $k\in\{1,2,3\}$ given by \prettyref{cor:pigeonhole} (on the complement of that event, choose $k=1$) and
define the sets $A=A_{k},B=B_{k}$ for that choice of $k$, following
\prettyref{eq:A-B-def}. With $\mathbb{P}$-probability going to $1$
as $N\to\infty$, \prettyref{eq:cap-to-band-ratio} of \prettyref{cor:sets-decay}
holds for such choice of $k$, independently of the realization of
the disorder and $N$: observe that the constants in \prettyref{cor:sets-decay}
are uniform over $k=1,2,3$, because the estimate of \prettyref{prop:cap-free-energy}
is uniform in all of the first $k=O(1)$ local minima.

We now use \prettyref{prop:conductance-bound} to upper bound the
spectral gap of $-\mathcal{L}$. To this end, let $r_{*}=\sqrt{N}\arccos(q_{*})$
and $r_{**}=\sqrt{N}\arccos(q_{**})$. Observe that $\mbox{Cap}(x,q_{**}-\epsilon N^{-1/2})$ is the ball of radius $r_{**}+\delta$ for a well chosen order one $\delta>0$. Its easily seen that for all small
$\delta$, $r_{**}+\delta$ is less than the injectivity radius of
$\mathcal{S}^{N}$\color{black}. Observe also that by \prettyref{cor:sets-decay}
and \prettyref{cor:pigeonhole}, we have that $\pi(A_{\delta}^{c})> 4\frac{\pi(B)}{\pi(A)}$
so that for large enough $N$, the conditions of \prettyref{prop:conductance-bound}
are satisfied. Applying that proposition then yields 
\begin{align*}
\lambda_{1}\leq & \frac{9\delta{}^{-2}\pi(A_{\delta}\backslash A)}{\pi(A)\pi(A_{\delta}^{c})-4\pi(A_{\delta}\backslash A)}=\frac{9\delta^{-2}\pi(B)}{\pi(A)\pi((A\cup B)^{c})-4\pi(B)}\,.
\end{align*}
Then \prettyref{cor:sets-decay} and \prettyref{cor:pigeonhole} together
imply that with $\mathbb{P}$-probability going to $1$ as $N\to\infty$,
\begin{align*}
\pi(A)\pi((A\cup B)^{c})-4\pi(B)\geq & \pi(A)(\tfrac{1}{2}-4\tfrac{\pi(B)}{\pi(A)})\\
\geq & \rho\pi(A)\,,
\end{align*}
for some sufficiently small but fixed $\rho>0$ (in particular, $\rho=\frac{1}{2}-\epsilon$
certainly works for large enough $N$). Then, we see that with $\mathbb{P}$-probability
approaching $1$ as $N\to\infty$, 
\begin{align*}
\lambda_{1}\leq & \frac{9\delta^{-2}\pi(B)}{\rho\pi(A)}\,,
\end{align*}
whence applying \prettyref{cor:sets-decay} again implies that there exists some 
$c_{1}(p,\beta)$, $c_{2}(p,\beta)>0$
such that 
\[\lim_{N\to\infty}\mathbb{P}\left(\lambda_{1}\leq c_{1}\exp(-c_{2}N)\right) =1\,. \qedhere
\]\end{proof}
With the above bounds in hand the proof of
item (1) of \prettyref{thm:spectral gap} is immediate.
\begin{proof}[\textbf{\emph{Proof of \prettyref{thm:spectral gap} part (1).}}]
 The lower bound is obtained in \prettyref{lem:lower-bound} and
the upper bound is obtained in \prettyref{lem:Upper-bound}.
\end{proof}

\subsection{High Temperature}

It remains to prove the lower bound on the spectral gap of $-\mathcal{L}$
at high temperatures. This follows straightforwardly from~\prettyref{lem:spectral-radius-bound}.
\begin{proof}[\textbf{\emph{Proof of \prettyref{thm:spectral gap} part (2) and \prettyref{prop:log-sob-ht}.}}]
 By \prettyref{lem:LS-implies-PI}, it suffices to prove \prettyref{prop:log-sob-ht}.
Recall that by the Curvature-Energy Balance theorem (\prettyref{prop:(Curvature-Energy-Balance-Lemma)}),
it suffices to show that there exists some $c>0$ such that the inequality,
\[
Ric_{\cS^{N}}(v,v)+\beta Hess(H)(v,v)\geq cg(v,v)
\]
holds uniformly over $\sigma\in\cS^{N}$ and $v\in T_{\sigma}\cS^{N}$
with probability tending to $1$. By scaling, it suffices to check
that this inequality holds for $v$ such that $g(v,v)=1.$ Recall
from \prettyref{fact:sphere-eigenvalue} that the Ricci tensor satisfies
\[
Ric_{\cS^{N}}=(1-\frac{1}{N})g\,.
\]
Thus it suffices to show that there is a constant $c$ such that with
probability tending to 1, we have 
\[
1-\frac{1}{N}+\beta Hess(H(\sigma))(v,v)\geq c\,.
\]

To see this, observe that by \prettyref{lem:spectral-radius-bound},
we have that on the complement of the event bounded there, with probability
going to $1$ as $N\to\infty$, 
\[
1-\frac{1}{N}+\beta Hess(H(\sigma)(v,v)\geq1-\frac{1}{N}-\beta C_{p}\,,
\]
holds for some constant $C_{p}>0$ . Choosing $\beta=\frac{\theta}{C_{p}}$
for any $\theta\in(0,1)$, we have that the righthand side is bounded
below by $1-\theta-o(1)$, yielding the inequality for $N$ sufficiently
large.
\end{proof}
\bibliographystyle{plain}
\bibliography{sphericalmixinglowtemp}

\end{document}